\documentclass[12pt]{article}
\usepackage[utf8]{inputenc}
\usepackage{latexsym,amsmath,amssymb,amsbsy,ulem,amsthm}
\usepackage[dvips]{graphicx}

\newtheorem{teor}{Theorem}
\newtheorem*{ter}{Theorem}
\newtheorem*{conj}{Conjecture}
\newtheorem{opr}{Definition}
\newtheorem{lem}{Lemma}
\newtheorem{prop}{Proposition}

\renewcommand\phi{\varphi}
\renewcommand\leq{\leqslant}
\renewcommand\geq{\geqslant}
\newcommand{\eps}{\varepsilon}
\newcommand{\ppmod}{\,{\rm{mod}}\,}

\begin{document}
\title{On the interpolation of integer-valued polynomials}
\author{V. V. Volkov, F. V. Petrov}
\maketitle
\begin{abstract}
It is well known, that if polynomial with rational 
coefficients of degree $n$ takes integer values 
in points $0,1,\dots,n$ then it takes integer values 
in all integer points. Are there sets of $n+1$ points with 
the same property in other integral domains? We show that 
answer is negative for the ring of Gaussian 
integers $\mathbb{Z}[i]$ when $n$ is large enough. 
Also we discuss the question about minimal possible size 
of set, such that if polynomial takes integer values in 
all points of this set then it is integer-valued.
\end{abstract}

\section{Introduction}

Let $R$ be integral domain and $f(x)$ be polynomial with 
coefficients from the quotient field of $R$ (or any larger field).
The polynomial $f$ is called \emph{integer-valued} 
if $f(R)\subset R$. There is much 
literature about integer-valued polynomials 
(see book \cite{kniga} and links therein). Sometimes 
polynomials are integer-valued when some weaker 
conditions holds, for example: if $R=\mathbb{Z}$ then 
polynomial of degree not higher than $n$ is integer-valued iff 
it takes integer values in points $0,1,\dots,n$. This fact 
is well-known and can be proved in many ways (one of them is 
direct interpolation of $f$ in above-metioned
points). Let us call 
set $M \subset R$ \emph{$n$-universal} if any 
polynomial $f$ of degree not higher than $n$ such 
that $f(M) \subset R$ is integer-valued. The natural question 
is: what is the minimal cardinality of $n$-universal 
set? Obviously, it is not 
less than $n+1$ for any infinite 
domain $R$. So there is special interest in 
studying situations when it is 
exactly $n+1$. Sequence $\{a_0,a_1,\dots\}$ with elements 
in $R$ is called simultaneous P-ordering 
if $\{a_0,a_1,\dots,a_n\}$ is an $n$-universal set for 
any $n$. Melanie Wood \cite{Wood} showed that there is no 
simultaneous P-ordering in the rings of integers of 
imaginary quadratic fields (in particular in 
the ring $\mathbb{Z}[i]$ of Gaussian integers).

From now on we consider domain 
$\mathbb{Z}[i]$. Let us call its elements integer 
numbers and irreducible elements --- prime numbers. 
Let $|n|$ be the absolute value of the complex 
number $n$, and $\|n\| = |n|^2$ --- its norm. 
For any non-zero $n\in \mathbb{Z}[i]$ there are 
exactly $\|n\|$ residues (remainders) 
modulo $n$. Let $|M|$ be the cardinality of 
finite set $M$. For $M\subset \mathbb{Z}[i]$, 
$\alpha\in \mathbb{Z}[i]\setminus 0$, $r\in \mathbb{Z}[i]$
let us denote
$$
M(r \ppmod \alpha):=M\cap \left(r+\alpha \mathbb{Z}[i]\right).
$$

The main results are following:

\begin{teor}\label{1}
There are no $n$-universal sets of cardinality 
$n+1$ in $\mathbb{Z}[i]$ provided $n$ is large enough.
\end{teor}

\begin{teor}\label{2}
 There exists an $n$-universal
set of cardinality $O(n)$ in $\mathbb{Z}[i]$.
\end{teor}

\section{Proof of the theorem \ref{1}}

Let
$C = \{c_1,c_2, \ldots, c_{n+1}\}$ 
be an $n$-universal set.

\begin{opr}
A \emph{volume} of the finite set of 
integers $X = \{x_1, x_2, \ldots, x_k\}$ is a product of 
all differences between elements of $X$:
$$V(X) = \prod_{1 \leq i < j \leq k} (x_i-x_j).
\footnote{Formally volume is defined only up to 
sign, but its choice is unimportant 
for the later considerations.}$$
\end{opr}

\begin{opr}
Non-negative integer numbers 
$a_1,a_2,\dots,a_N$ are called \emph{almost equal} if any 
two of them are either equal or differ by one. 
Finite set $M \subset \mathbb{Z}[i]$ is
\emph{almost uniformly distributed modulo} non-zero number 
$\alpha\in \mathbb{Z}[i]$  if 
numbers $|M(r \ppmod \alpha)|$ (where $r$ runs over 
all possible remainders modulo $\alpha$) are almost equal. 
\end{opr}

We are going to use the following standard facts:

\begin{prop}
1) There are exactly one way to split non-negative 
integer $N$ into sum of $k$ almost equal numbers.

2) This and only this partition minimizes 
sum of squares of summands.

3) If we split $N$ into $k$ almost equal parts, and 
each of them into $\ell$ almost equal parts, then we get 
the partition of $N$ into $kl$ almost equal parts.
\end{prop}

\begin{lem}
\label{vol} 
$C=\{c_1,c_2, \ldots, c_{n+1}\}$ is an $n$-universal set 
if and only if $C$ is almost uniformly distributed
modulo $p^k$ for each prime power $p^k$. In that 
case volume $V(C)$ is minimal (by absolute value) 
among all volumes of the sets 
consisting of $n+1$ integer points, and 
moreover $V(C)$ divides volume $V(B)$ for 
any other set $B$. If also $|V(B)| = |V(C)|$ then 
set $B$ is $n$-universal too.
\end{lem}

\begin{proof}
For arbitrary $c_m \in C$ let us consider 
the polynomial $Q_m(x)$ of degree $n$, which 
takes value $0$ in all points $c_i$ (where $i \ne m$) and 
value $1$ in the point $c_m$. Let $f$ be a polynomial 
of degree not greater than $n$. By Lagrange 
interpolation formula we have that 
$$f(x)=\sum f(c_k) Q_k(x).$$ So $C$ is an $n$-universal 
set if and only if all polynomials $Q_m$ are integer-valued. Also
\begin{equation}
\label{lej}
Q_m(x) = \prod_{i\ne m} \frac{(x-c_i)}{(c_m-c_i)}.
\end{equation} 
Polynomial $Q_m$ is an integer-valued if 
and only if for any prime $p$ maximal 
degree of $p$ that divides numerator 
(where $x$ is a fixed integer) is not 
less than maximal degree of $p$ dividing 
denominator. Now let us fix $p$.

Assume that elements of $C$ is almost uniformly distributed
modulo every power of $p$. Then for any power $p^k$ 
there are at least as many multiplies that divides 
$p^k$ in the numerator as in the denominator. Indeed, there are exactly 
$r-1$ such multiplies in the denominator 
(where $r=|C(c_m \ppmod p^k)$). And numerator 
has exactly $|C(x \ppmod p^k)|$ such multiplies, which is not 
less than $r$ under our assumption. After 
summing by $k=1,2,\dots$ one gets the above-mentioned 
condition on the maximal degree of $p$ in the numerator and denominator.

Now assume that elements of $C$ is not almost uniformly
distributed $p^k$ for some $p^k$, and 
let $k$ be minimal among such numbers. Then for 
some $r$ the set $C(r \ppmod p^{k-1})$ is not
almost uniformly distributed modulo $p^k$  
(the contrary would contradict to the part 
3 of Proposition 1). Let 
$|C(c_m \ppmod p^k)|\geq |C(r \ppmod p^k)|+2$ 
and  $c_m\equiv r$ $\ppmod p^{k-1}$. Without loss of 
generality $|C(c_m \ppmod p^{k+1})|>|C(r \ppmod p^{k+1})|$ 
(otherwise we may replace $c_m$ with 
some $c_{m'}$ and $r$ with some $r'$ with the same 
remainders modulo $p^k$ and suitable remainders 
modulo $p^{k+1}$). Similarly we may think that 
$$
|C(c_m \ppmod p^{k+\rho})|>|C(r \ppmod p^{k+\rho})|, \rho=1,2,\dots
$$
Consider the formula of $Q_m(r)$. If $l=1,2,\dots,k-1$ then 
there are equal number of multiplies divisible by $p^l$ in 
the numerator and denominator. If $l=k$ then denominator 
has more such multiplies than numerator, 
and if $l > k$ then not less than numerator. Summing 
by $l$ we get that power of $p$ in the 
numerator is less than in the denominator, so $Q_m$ is not 
integer-valued.

First part of the lemma is now proved. Note that 
if $C$ is almost uniformly distributed modulo $q=p^k$  
then number of pairs in $C$ with equal remainders 
modulo $q$ is minimal among all $n+1$-element sets. Last 
follows immediately from the part 2 of Proposition 1. 
Summing by $k=1,2,\dots$ we get that power 
of $p$ in the $V(C)$ is minimal if and only 
if $C$ is almost uniformly
distributed modulo $p^k$ for any $k$. 
This proves the second and third parts of the lemma. 
\end{proof}

Note that for any particular $p$ it is 
not hard to construct 
the set, which is almost uniformly distributed 
modulo $p^k$ (for any $k$). However ``to combine'' all 
these sets into one (which fits all $p$) is not always possible.

For example, for $n=1,2,3,5$ there are universal sets $\{0,1\}$, $\{0,1,i\}$, $\{0,1,i,1+i\}$, $\{0,1,2,i,1+i,2+i\}$.
But, as one may easily check, there is no $4$-universal set with $5$ elements. 

%Skipped checking here.

Now let $C=\{c_1,c_2,\dots,c_{n+1}\}$ be an $n$-universal set.

Our next aim is to
replace set $C$ with more convenient set without 
increasing absolute value of its volume. 
We are going to symmetrize set $C$ with respect 
to vertical and horizontal lines of integer 
lattice. We need the following combinatorial lemma for that.

\begin{lem}
\label{comb}
Let $A$ and $B$ be some finite
sets of integer 
rational points of prescribed cardinalities.
Fix number $k \in \mathbb{N}_0$. 
Then number of pairs $a \in A$, $b \in B$ such that 
$|a-b| \leqslant k$ is minimal when each of sets 
$A$ and $B$ is a set of consecutive integer 
points and midpoints of segments formed by $A$ 
and $B$ are equal (in case of 
$|A| \equiv |B| \; (\mathrm{mod }\; 2)$) or differ 
by $\frac{1}{2}$ (otherwise).
\end{lem}

\begin{proof} 
Replace each point $b \in B$ by the segment $b'$ with 
length $2k$ and midpoint in $b$. Then 
$|a-b| \leq k$ exactly when $a \in b'$.

Now we are going to modify set of segments $B'$ 
without decreasing the number of mentioned pairs, 
until segments becomes consecutive. Let 
the leftmost segment start in the point $0$, let there 
also be some segments that start in the 
points $1, \ldots, s-1$, but no segment starts 
in the point $s$. Move each of mentioned 
segments one to the right. Let $Z_-$ be a 
set of integer points which now are 
cowered by one less segments 
from $B'$, $Z_+$ --- by one more 
segments from $B'$. All other points are
covered by 
the same number of segments as before.

Then $|Z_-| = |Z_+| = \min (s, 2k+1) = z$. 
Denote $A_i = A\cap Z_i$, $|A_i| = m_i$, where 
$i\in\{-, +\}$. If $m_- \leq m_+$ then 
above-mentioned shift does not decrease 
the number of interesting pairs. Otherwise, 
denote $m = m_- - m_+ > 0$. Consider $m$ leftmost 
points in $A_-$. They were covered (totally) by 
at most $t = (z-m_-+1) + (z-m_-+2) + \ldots (z-m_+)$ segments.
Consider $m$ leftmost points in $Z_+ \setminus A_+$. 
They are now covered (totally) by at least $t$ segments. 
So, if in addition to shift we replace 
these $m$ points in $A$ with the 
mentioned $m$ points in $Z_+ \setminus A_+$, then 
number of interesting pairs does not decrease.

If there was at least one segment 
in $B'$ apart from shifted ones, then sum of 
pairwise distances between midpoints of all segments 
in $B'$ decreased. So in the end of our process the segments 
in $B'$ must be consecutive. Obviously, after that, 
points of $A$ should be placed in the way described in the statement.
\end{proof}

For each vertical line $\{x=c\}$ replace
the points of $C$ on this line to
the segment centered either in $(c,0)$ 
or $(c,1/2)$ (depending on parity of 
the number of points).

Using Lemma $\ref{comb}$ to all numbers 
$k \in \mathbb{N}_0$ and all pairs of vertical lines 
 we see that this operation does not increase the 
absolute value of the volume. After that one 
may act in the same way with the horizontal lines.

Hereupon set $C$ becomes rather ``symmetric''. In 
particular, if $n_1$ is the length of the segment in 
the intersection $C$ and real axis, $n_2$ --- same for 
the imaginary axis, then set $C$ is placed 
inside a rectangle $n_1 \times n_2$.

We need the following well-known fact for the later reasoning:

\begin{ter}
Quotient of two consecutive rational prime numbers of 
the form $4k+3$ tends to $1$.
\end{ter}
  
For particular $\eps > 1$ there always exists 
prime numbers of the form $4k+3$: $p \in (\sqrt{n},\eps\sqrt{n})$ and 
$q \in \left(\sqrt{\dfrac{n}{2}}, \eps\sqrt{\dfrac{n}{2}}\right)$ 
when $n$ is large enough. These numbers are also primes in
$\mathbb{Z}[i]$. But $\|p\| > n$, so elements 
of $C$ cannot give equal remainders by 
modulo $p$. In particular $n_1 < p$ and $n_2 < p$. So 
set $C$ is placed inside some 
square $(p-1)\times(p-1)$. Moreover 
$2\|q\| > n$, so elements of $C$ may give each 
remainder by modulo $q$ at most twice. It is easy to 
see that in this case there are no more 
than $p^2 - 2(p-q)^2$ points in $C$. So
\begin{align*} 
n &< p^2 - 2(p-q)^2 \\
n &< -p^2 + 4pq -2q^2 \\
p^2 + n + 2q^2 &< 4pq \\
3n &< 2\sqrt{2}\eps n \\
\frac{3}{2\sqrt{2}} &< \eps. 
\end{align*}

Taking $\eps$ small enough one gets a contradiction. 
Theorem \ref{1} is now proved.

\section{Examples of universal sets}

Let us prove Theorem \ref{2}.

Let $p$ be a prime number. Set $C$ of cardinality 
$n+1$ is called \emph{($n, p$)-universal} if 
it is almost uniformly
distributed modulo $p^k$ (for any $k$). 
Proof of Lemma \ref{vol} shows that if 
polynomial $f$ of degree at most $n$ takes integer 
values in all points of some $(n,p)$-universal set, 
then its value in any integer point is a fraction 
without $p$ in the denominator. Hence set 
that contains $(n,p)$-universal subset 
for any $p$ is always $n$-universal.

Now we are going to prove that integer points 
inside the circle with center in $0$ and radius 
$R$ contains at least $cR^2$ points 
that form $(n, p)$-universal set (here $c$ is some 
absolute constant). From this we conclude that integer points 
inside the circle $\sqrt{\frac{n+1}c}$ forms an 
$n$-universal set as required.

Let $|p^k|\leq R\sqrt{2}<|p|^{k+1}$.

Consider all integers that lies in half-open square 
$$
\{p^k(a+bi), -1/2-A\leq a,b< 1/2+A\}
$$
for some non-negative rational integer $A$.

Clearly each remainder modulo $p^k$ 
occurs exactly  $(2A+1)^2$ times among considered points.
Choose $A$ in such a way that

(i) considered points lies inside a circle 
with the center $0$ and radius $R$;

(ii) no two of considered points have 
equal remainders by modulo $p^{k+1}$.

Property (i) holds if $(1/2+A)|p^k|\leq R$, property 
(ii) definitely holds if diagonal of the 
square is shorter than $p^{k+1}$, i. e. 
if $(2A+1)\sqrt{2}|p^k|\leq |p^{k+1}|$, 
or $2A+1\leq |p|/\sqrt{2}$. Note that $A = 0$ satisfies 
both conditions. Also note that two consecutive 
half-integer rational numbers differs in no more than 
three times. Choose maximal $A$ that satisfies 
both inequalities. We have that 
$$
A+1/2\geq \frac13\min(R/|p^k|,\frac{|p|}{2\sqrt{2}})
\geq \frac{R}{6|p^k|}.
$$
Number of selected points is $(2A+1)|p^k|^2\geq R^2/9$ as required.

Theorem \ref{2} is proved.

\vskip 0.2cm
The following conjecture generalizes both our theorems.
\begin{conj} The minimal cardinality of the $n$-universal set
in $\mathbb{Z}[i]$ grows as $\frac{\pi}{2}n+o(n)$ and 
asymptotically sharp example is realized on the set of integer
points inside the circle of radius $\sqrt{n/2}+o(\sqrt{n})$. 
\end{conj}
Note that the points inside less circle do not form $n$-universal
set for sure (because the points inside circle of radius $R$ give
not all possible
residues modulo some prime of absolute value slightly
greater then $R\sqrt{2}$). 

It also looks probable that the theorems, analagous to proven
here for $\mathbb{Z}[i]$, hold for wide class of integral
domains.

\textbf{Acknowledgement.} We are grateful 
to prof. Paul-Jean Cahen
for useful remarks and conslutations.


\begin{thebibliography}{99}
 \bibitem{kniga} Paul-Jean Cahen, Jean-Luc Chabert. 
\textit{Integer-valued polynomials.} AMS, 1997.
\bibitem{Wood} Melanie Wood. \textit{P-orderings: a metric 
viewpoint and the non-existence of simultaneous
orderings}. J. Number Theory, 99 (2003), 36-56.

\end{thebibliography}
\end{document}